\documentclass[11pt,a4paper,oneside]{article}

\usepackage[T1]{fontenc}
\usepackage[cp1250]{inputenc} % kodowanie: opcjonalnie [Latin-2] (ISO-8859-2)
\usepackage{times}

\usepackage[leqno]{amsmath} % pakiet ams (numeracja wyrównywane do lewej)
\usepackage{amsthm,amsfonts} % symbole do pakietu ams + pakiet definiowania twierdzeń (można dodać amssymb)
\usepackage{bm} % pogrubione greckie litery

\input xy  \xyoption{all}

\usepackage{geometry}
\usepackage[usenames]{color} %obsługa kolorów

\usepackage{graphicx} % do wstawiania obrazkow % opcja: draft - tylko ramki

\usepackage{enumerate} %pozwala bawić się wypunktowaniem
\usepackage{hyperref} % hiperodnośniki (ładne są po przerobieniu na pdf.)

\newcommand{\D}{\mathcal{D}} %SR distribution
\newcommand{\E}{\mathcal{E}} %energy functional
\DeclareMathAlphabet{\mathpzc}{OT1}{pzc}{m}{it}
\newcommand{\F}{F^{u(\cdot)}}
\newcommand{\f}{X}
\newcommand{\Dperp}{\dot\gamma^\perp}

\newcommand{\R}{\mathbb{R}}
\newcommand{\Id}{\operatorname{Id}}
\newcommand{\wt}[1]{\widetilde{#1}}
\newcommand{\wh}[1]{\widehat{#1}}
\newcommand{\vect}{\operatorname{vect}}
\newcommand{\T}{\mathrm{T}}
\newcommand{\lra}{\longrightarrow}
\newcommand{\ra}{\rightarrow}
\newcommand{\eps}{\varepsilon}
\newcommand{\dd}{\operatorname{d}}
\newcommand{\pa}{\partial}
\def\<#1>{\big\langle #1\big\rangle}
\def\(#1){\left( #1\right)}

\numberwithin{equation}{section} %zmienia numerację równań

\theoremstyle{plain} 
\newtheorem{theorem}{Theorem}[section]

\newtheorem{lemma}[theorem]{Lemma}

\theoremstyle{definition}
\newtheorem{definition}[theorem]{Definition}

\theoremstyle{remark}

\title{Why are normal sub-Riemannian extremals locally minimizing?}
\author{Micha\l{} J\'{o}\'{z}wikowski\thanks{Research was supported by the Polish National Science Center under the grant DEC-2012/06/A/ST1/00256.}\\
{\small email: mjozwikowski@gmail.com}\\[0.2cm]
Institute of Mathematics,\\
 Polish Academy of Sciences\\[0.5cm]
 Witold Respondek\thanks{Research was supported by the Polish National Science Center under the grant DEC-2011/02/A/ST1/00208.}\\
{\small email: witold.respondek@insa-rouen.fr}\\[0.2cm]
Normandie Universit\'{e}, France\\
 INSA de Rouen, Laboratoire de Math\'{e}matiques}

\date{\today}

\begin{document}
\maketitle
\begin{abstract}
It is well-known that normal extremals in sub-Riemannian geometry are curves which locally minimize the energy functional. Most proofs of this fact do not make, however, an explicit use of relations between local optimality and the geometry of the problem.
In this paper, we provide a new proof of that classical result, which gives insight into direct geometric reasons of local optimality. Also the relation of the regularity of normal extremals with their optimality becomes apparent in our approach.
\end{abstract}

\paragraph*{Keywords:} sub-Riemannian geometry, normal extremal, local optimality, geodesic

\paragraph*{MSC 2010:} 49K15, 53C17, 58A30
%58A30 Vector distributions (subbundles of the tangent bundles)
%53D10 Contact manifolds, general
%53C17 Sub-Riemannian geometry
%49Kxx Optimality conditions
	%49K15 Problems involving ordinary differential equations

%%%%%%%%%%%%%%%%%%%%%%%%%%%%%%%%%%%%%%%%%%%%
%%%%%%%%%%%%%%%%%%%%%%%%%%%%%%%%%%%%%%%%%%%%
\section{Introduction}\label{sec:intro}	

\paragraph{Motivations.} Our initial motivation to undertake this research comes form our recent study \cite{MJ_WR_contact_pmp}. In that paper we developed a contact geometry approach to optimal control problems. As a particular application we discussed sub-Riemannian (\emph{SR}, in short) geodesic problems and obtained elegant geometric characterizations of both normal and abnormal SR extremals. In particular, \emph{normal SR extremals} (\emph{NSRE}s, in short) can be nicely described in terms of the distribution orthogonal to the actual extremal and the flow related with the optimal control. That result, here formulated as Theorem \ref{thm:normal}, was first obtained in \cite{Alcheikh_Orro_Pelletier_1997}. Clearly, NSREs satisfy only first-order conditions for optimality, yet it is well-known that these suffice for their local optimality. Thus it is natural to ask the following question:
\begin{equation*}
\label{eqn:question}\tag{Q}
\text{How does the geometry of NSREs reflect their local optimality?}
\end{equation*}
In this paper we answer question \eqref{eqn:question} providing a detailed geometric proof of local optimality of NSREs (Theorem \ref{thm:main}). Our original idea, which allowed to relate the geometric characterization of NSREs with their local optimality, is to study homotopies (and the related variations) of SR trajectories. We discuss the details of our approach in the next paragraph.

The standard proofs of the local optimality of NSREs (see, e.g., \cite{Liu_Sussmann_1995, montgomery2006tour}) are deeply rooted in symplectic geometry. Usually one starts with the Hamiltonian description of a NSRE $\gamma$ provided by the Pontryagin Maximum Principle \cite{Pontr_Inn_math_theor_opt_proc_1962}. The basic idea is to construct a solution of a Hamilton-Jacobi equation in a neighborhood of $\gamma$, and then to use this solution to build the so-called \emph{calibration} of $\gamma$, i.e., a closed 1-form which estimates the SR length from below with an equality on $\gamma$ (see Sec. 1.9 of \cite{montgomery2006tour} and for another proof using similar arguments \cite{Rifford_2014}).
To prove the local optimality of the considered NSRE we integrate the calibration over a closed contour containing $\gamma$ and use Stokes theorem. That method is very elegant and powerful although it does not use explicitly the geometric reasons behind the local optimality of NSREs. We believe that the new proof presented in this work, an using homotopies, gives an insight into these reasons. For example, the relation between optimality and regularity for NSREs is apparent in our approach. In our proof we do not use symplectic geometry, working directly with the SR distribution and the SR metric, and use only basic differential-geometric tools. We hope that similar ideas will allow to study the optimality issues for other classes of optimal control problems.

Let us note that another point of view on NSREs can be found in \cite{Bonnard_Caillau_Trelat_2007}, where local optimality of smooth extremals (for a general optimal control problem) is discussed from the perspective of the second-order necessary optimality conditions.

%%%%%%%%%%%%%%%%%%%%%%%%%%%%%%%%%%%%%%%%%%%
\paragraph{Outline of the proof.} Let us now sketch the main steps of the proof and explain geometric reasons which make NSREs locally optimal.

First observe that since we are interested in local optimality only, without any loos of generality we can reformulate the (local) SR geodesic problem as an optimal control problem on $\R^n$ quadratic in cost and linear in controls.

Naively, to prove that a given NSRE $\gamma_0:[0,T]\ra \R^n$ (corresponding to a control $u(t)$) is optimal, one should show that any other trajectory of the system sharing the same end-points as $\gamma_0$ has energy (length) bigger than that of $\gamma_0$. It is, however, quite difficult to describe SR trajectories with given end-points, so instead let us consider any SR trajectory $\gamma_1:[0,T]\ra \R^n$ (corresponding to a control $u(t)+\Delta u(t)$) for which we assume only that it has the same initial point as $\gamma_0(\cdot)$ and that the energy of $\gamma_1(\cdot)$ is smaller than the energy of $\gamma_0(\cdot)$. Showing that the end-point $\gamma_1(T)$ is necessarily different from $\gamma_0(T)$ will end the proof.

In order to compare these end-points, we extend $\gamma_0(\cdot)$ and $\gamma_1(\cdot)$ to a natural homotopy $\gamma_s:[0,T]\ra \R^n$ with $s\in[0,1]$, simply by considering SR trajectories corresponding to the intermediate controls $u(t)+s\Delta u(t)$ and all sharing the same initial point $\gamma_0(0)$.  Now we shall  concentrate our attention on the variation of the family $\gamma_s(\cdot)$ with respect to the parameter $s$:
$$b_s(t):=\pa_s\gamma_s(t)\ .$$
It is clear that the vector $b_0(T)$ approximates at $\gamma_0(T)$ the curve of end-points $\gamma_s(T)$ joining $\gamma_0(T)$ and $\gamma_1(T)$.

It turns out that it is easy to obtain an analytic formula for $b_0(T)$ (Lemma \ref{lem:b_0}) and that the geometric characterization of NSREs (Theorem \ref{thm:normal}) yields a specific behavior of this vector -- it has to point ``much'' backward with respect to $\dot\gamma_0(T)$ -- see Figure \ref{fig:idea}. In other words, the fact that a given SR trajectory $\gamma_0(\cdot)$ is a NSRE implies that the end-points $\gamma_s(T)$ of the natural homotopy joining it with $\gamma_1(\cdot)$ has to wind backward along $\gamma_0(\cdot)$ (provided that the energy of $\gamma_1(\cdot)$ is not greater than that of $\gamma_0(\cdot)$). If $T$ is sufficiently small, then the curve $\gamma_s(T)$ has ``no time'' to return to $\gamma_0(T)$, enforcing $\gamma_0(T)\neq\gamma_1(T)$ and thus proving the optimality of $\gamma_0(\cdot)$ and, in particular, uniqueness.
\medskip

Clearly, the above geometric understanding has to be backed up with some quantitative estimations to guarantee a rigorous proof.  For this purpose we prove elementary Lemmata \ref{lem:basic_estimate} and \ref{lem:b_0_nsre}, which show that the norm $\|b_0(T)\|$ is of order $\|\Delta u(\cdot)\|^2_{L^2[0,T]}$, and Lemma \ref{lem:estimate_delta_b}, estimating the norm $\|b_s(T)-b_0(T)\|$ by an expression of order $T\|\Delta u(\cdot)\|^2_{L^2[0,T]}$. If $T$ is small enough, these two estimates guarantee that the behavior of $b_0(T)$ resembles the behavior of $\gamma_s(T)$ (which is an integral curve of $b_s(T)$). The proof of Lemma \ref{lem:estimate_delta_b} is also elementary, yet quite long. We basically use standard analytic tools (Gr\"{o}nwall's inequality and the mean value theorem) threating the end-point variation $b_s(T)$ as a solution of an $s$-dependent ODE.
\medskip

From a more abstract perspective we believe that it is instructive to relate our proof with the view of optimal control formulated in \cite{Agrachev_Sachkov_2004}, where one should view controlled trajectories as belonging to the extended configuration space containing both configurations and costs (this point of view is actually a foundation of the Pontryagin Maximum Principle \cite{Pontr_Inn_math_theor_opt_proc_1962}).  Now basically, a trajectory of a control system is optimal if its end-point lies on the boundary of the set of points reachable from a given intital configuration in a given time. For a NSRE $\gamma_0(\cdot)$ we simply prove that points $(\gamma_0(T),e)$, where the energy $e$ is smaller than the energy $e_0$ of $\gamma_0(\cdot)$ do not belong to the set reachable from $(\gamma_0(0),0)$ in time $T$, for $T$ sufficiently small.
\begin{figure}%
\centering
  \def\svgwidth{0.5\textwidth}
 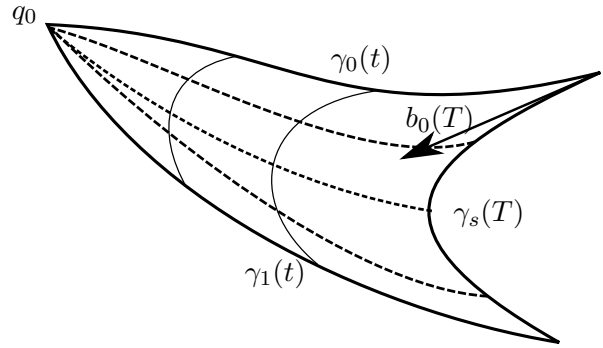
\caption{The geometric reason for local optimality of NSREs: if the cost of $\gamma_1(\cdot)$ is smaller than the cost of $\gamma_0(\cdot)$, then the end-points $\gamma_s(T)$ of the natural homotopy joining the NSRE $\gamma_0(\cdot)$ with $\gamma_1(\cdot)$ move initially backward along $\gamma_0(\cdot)$ (in the direction of vector $b_0(T)$). If $T$ is sufficiently small, points $\gamma_s(T)$ have ``no time'' to go back to $\gamma_0(T)$.}%
\label{fig:idea}%
\end{figure}

\paragraph{Organization of the paper.} In Section \ref{sec:nsre}, we introduce the SR geodesic problem and reformulate its local version in the framework of the optimal control theory. Later we give a geometric characterization of NSREs in Theorem \ref{thm:normal}. We end that part by defining local optimality and stating the main result in Theorem \ref{thm:main}.

In Section \ref{sec:homotopy}, we prove Theorem \ref{thm:main} according to the scheme sketched above. We start with basic estimates related with the energy comparison in Lemma \ref{lem:basic_estimate}. Then we define a natural homotopy between two trajectories of the system and define the related variation. We translate conditions characterizing NSREs to the language of such variations in Lemmata \ref{lem:b_0} and \ref{lem:b_0_nsre}. This, together with natural estimates on the variations from Lemma \ref{lem:estimate_delta_b}, suffices to prove Theorem \ref{thm:main}.

Appendix \ref{app:technical} is devoted to technical and supplementary results. In Lemma \ref{lem:flows} we formulate basic properties of the flows of time-dependent vector fields. Then, by a series of estimates (Lemmata \ref{lem:estimate_delta_q}--\ref{lem:estimate_b}) based on Gr\"{o}nwall's inequality (Lemma \ref{lem:gronwall}), we prove Lemma \ref{lem:estimate_delta_b}

\paragraph{Notation and conventions.} In the paper we use standard notations and conventions of differential geometry and control theory. By $\|y\|_p$ we shall denote the $p$-norm of a vector $y=(y_1,\hdots,y_n)\in \R^n$, i.e., $\|y\|_p:=\left((y_1)^p+\hdots+(y_n)^p\right)^{\frac 1p}$. Given a map $y:[0,T]\ra\R^n$, the expression $\left(\int_0^T\|y(t)\|_p^p\dd t\right)^{\frac 1p}$ gives the $L^p$-norm of $y(\cdot)$, and is denoted by $\|y(\cdot)\|_{L^p[0,T]}$.
	
We say that a curve $\gamma:[0,T]\ra\R^n$ is \emph{absolutely continuous with bounded derivative} (\emph{ACB}, in short) if it is absolutely continuous and its derivative (well-defined a.e.) is essentially bounded. 

%%%%%%%%%%%%%%%%%%%%%%%%%%%%%%%%%%%%%%%%%%%%	
%%%%%%%%%%%%%%%%%%%%%%%%%%%%%%%%%%%%%%%%%%%%	
\section{Normal sub-Riemannian extremals}\label{sec:nsre}	

\paragraph{Sub-Riemannian geodesic problem.} We shall be working within the following geometric setting. Let $Q$ be a smooth $n$-dimensional manifold and let $\D\subset\T Q$ be a smooth  distribution of rank $k$ on $Q$. By $g:\D\times_Q\D\ra\R$ we shall denote a positively defined symmetric bilinear product on $\D$ (called a \emph{sub-Riemannian metric}). The triple $(Q,\D,g)$ constitutes a \emph{sub-Riemannian} (\emph{SR}, in short) \emph{structure}. Given a pair of points $q_0,q_1\in Q$ and an interval $[0,T]\subset\R$ it is natural to consider the following \emph{SR geodesic problem}:
\begin{equation*}\tag{P}
\label{eqn:sr_geod_prob}
\text{Find an ACB curve $\gamma:[0,T]\ra Q$ satisfying the following conditions:}
\end{equation*}
\begin{enumerate}
	\item\label{cond:1} $\gamma(\cdot)$ is tangent to $\D$, i.e., $\dot \gamma(t)\in \D_{\gamma(t)}$ for $t\in[0,T]$ a.e..
	\item\label{cond:2} $\gamma(\cdot)$ joins $q_0$ and $q_1$, that is, $\gamma(0)=q_0$ and $\gamma(T)=q_1$.
	\item $\gamma(\cdot)$ realizes the minimum of the energy functional 
	$$\gamma(\cdot)\longmapsto \E(\gamma(\cdot)):=\frac 12 \int_0^T g(\dot \gamma(t),\dot \gamma(t))\dd t$$
	in the set of all curves satisfying conditions \ref{cond:1} and \ref{cond:2} above.
\end{enumerate}
A solution of \eqref{eqn:sr_geod_prob}, for given points $q_0$ and $q_1$, is called a \emph{minimizing SR geodesic}. Geometrically, SR geodesics are curves $\gamma(\cdot)$ that minimize the length functional $\gamma(\cdot)\mapsto l(\gamma(\cdot)):=\int_0^T\left(g(\dot\gamma(t),\dot\gamma(t))\right)^{\frac 12}\dd t$ but minimization of both problems coincide up to a reparametrization (see comments preceding the statement of Theorem \ref{thm:normal}). 

In a local version, problem \eqref{eqn:sr_geod_prob} can be formulated as an optimal control problem. Indeed, choose  $k$ smooth, linearly independent vector fields $X_1,\hdots, X_k$ spanning locally the distribution $\D$.\footnote{Such a choice is always possible in a single coordinate domain of a local trivialization of the bundle $\D$.} Note that by applying the Gram-Schmidt orthogonalization, without any loss of generality, we may assume that the fields $\{X_i\}_{i=1,\hdots,k}$ form a $g$-orthonormal basis of $\D$. Now every ACB curve $\gamma(t)$ tangent to $\D$ may be considered as a trajectory of the linear control system 
$$\f:Q\times\R^k\lra \T Q\qquad (q,u^i)\longmapsto \f(q,u)=\sum_{i=1}^k u^iX_i(q)$$
corresponding to some \emph{control}, i.e., a locally bounded measurable curve $u(t)=(u^1(t),u^2(t),\hdots,u^k(t))\in\R^k$. In other words, $\gamma(t)$ is an integral curve of a time-dependent vector field $\f_{u(t)}$ obtained by evaluating the map $\f$ on $u(t)$, i.e.,
$$\dot \gamma(t)=\f_{u(t)}(\gamma(t)):=\sum_iu^i(t)X_i(\gamma(t))\quad \text{for $t\in[0,T]$ a.e..}$$
Under the assumption that the fields $X_i$ are normalized as $g(X_i,X_j)=\delta_{ij}$, the energy related with such a $\gamma(\cdot)$ is simply
$$\E(\gamma(\cdot))=\frac 12\int_0^T\sum_i \left(u^i(t)\right)^2\dd t=\frac 12\|u(\cdot)\|_{L^2[0,T]}^2\ .$$

%%%%%%%%%%%%%%%%%%%%%%%%%%%%%%%%%%%%%%%%%%%%%%%%
\paragraph{Normal extremals.}	
After formulating a local version of the SR geodesic problem as an optimal control problem it is possible to look for its solution by means of the Pontryagin Maximum Principle (see \cite{Pontr_Inn_math_theor_opt_proc_1962} and \cite{Agrachev_Barilari_Boscain_2012,Agrachev_Sachkov_2004,montgomery2006tour,Rifford_2014}). It is well known that this theorem provides two classes of candidates for solutions, called \emph{normal} and \emph{abnormal} \emph{extremals}, respectively. In this paper we restrict our attention to the first class of such candidates. We shall call them \emph{normal sub-Riemannian extremals} (\emph{NSRE}s, in short). In our considerations we are going to apply a less used geometric characterization of NSREs obtained for the first time in \cite{Alcheikh_Orro_Pelletier_1997} and recently in our paper \cite{MJ_WR_contact_pmp} by another method. In order to formulate that result we will need to introduce a few geometric notions. 

Let $\gamma:[0,T]\ra Q$ be a trajectory of the considered system related with a control $u(t)$ and such that $\gamma(0)=q_0$. Recall that a particular choice of $u(t)$ defines a time-dependent vector field $\f_{u(t)}$ on $Q$. 
For any $t, \tau\in[0,T]$, by $\F_{t\tau}$ we shall denote the \emph{time-dependent flow}\footnote{Calling it a TD-flow should not lead to confusions, although it is not a (local) flow.} (\emph{TD-flow}, in short) of this vector field from time $\tau$ to time $t$. That is, by definition $t\mapsto \F_{t\tau}(q)$ is the integral curve of the considered vector field $\f_{u(t)}$ at time $t$ with the initial condition $q$ at time $t=\tau$. In particular, $\gamma(t)=\F_{t0}(q_0)$. The problem of regularity of such a family of maps $\F_{t\tau}$ was discussed in our previous study (\cite{MJ_WR_contact_pmp}, Lemma~2.1). Here we recall it in Lemma \ref{lem:flows}. In short, if a control $u(t)$ is locally bounded and measurable, then the related $\F_{t\tau}$ (for a fixed $\tau$) is a well-defined family of local diffeomorphisms continuously depending on the variable $t$.   

To characterize NSREs the following object will be of special importance 
$$\F_\bullet(\Dperp)_{\gamma(t)}:=\vect_\R\left\{\T \F_{t\tau}(X)\ |\ X\in \D_{\gamma(\tau)},\ g(X,\dot \gamma(\tau))=0,\ \tau\in[0,T]\right\}\ .$$
Geometrically $\F_\bullet(\Dperp)_{\gamma(t)}$ is the smallest (with respect to the inclusion) distribution along the trajectory $\gamma(t)$ which contains the part of $\D$ which is $g$-orthogonal to $\dot\gamma(t)$ and at the same time is invariant with respect to the TD-flow $\F_{t\tau}$. 

Note that it follows easily from the Cauchy-Schwarz inequality that among different parametrizations of a given curve defined on a fixed time interval $[0,T]$, the one with the uniform parametrization (i.e., $g(\f_{u(t)},\f_{u(t)})^{\frac 12}=\|u(t)\|_2=\operatorname{const.}$) has the smallest energy. Since a global rescaling of the metric $g$ by a constant has no impact on the solutions of the SR geodesic problem, without any loss of generality we may assume that this constant equals to 1. To sum up, when considering the solutions of the SR geodesic problem, without any loss of generality we may assume that the control $u(t)$ is normalized by $\|u(t)\|_2=1$.

Now we are ready to formulate a geometric characterization of NSREs. 
\begin{theorem}[Characterization of NSREs, \cite{Alcheikh_Orro_Pelletier_1997,MJ_WR_contact_pmp}]
\label{thm:normal}
Fix a normalized control $u(t)$,  i.e.,\\ $\left(g(\f_{u(t)},\f_{u(t)})\right)^{\frac 12}=\|u(t)\|_2\equiv 1$ and the corresponding trajectory $\gamma(t)$, $t\in[0,T]$. Then the following are equivalent: 
\begin{enumerate}[(a)]
	\item\label{cond:A_normal} The  trajectory $\gamma(t)$ is a NSRE.
	\item\label{cond:B_normal} The trajectory $\gamma(t)$ satisfies the following two conditions:
	\begin{enumerate}[(b1)]
		\item \label{cond:B_normal1} The velocity $\dot\gamma(t)=\f_{u(t)}(\gamma(t))$ is of class ACB with respect to $t$.
	  \item \label{cond:B_normal2} The distribution  $\F_\bullet(\Dperp)_{\gamma(t)}$
	does not contain the velocity $\dot\gamma(t)=\f_{u(t)}(\gamma(t))$ for any $t\in[0,T]$.
	\end{enumerate}	
\end{enumerate}
\end{theorem}
In other words, to know if a given trajectory $\gamma(t)$ is a NSRE it is enough to check whether it is $C^1$-smooth with an ACB derivative and then to take along $\gamma(t)$ the part of $\D$ that is $g$-orthogonal to $\dot\gamma(t)$, produce from it new directions by the action of the TD-flow $\F_{t\tau}$ and then ensure that this procedure does not generate a direction tangent to $\gamma(t)$. Under additional technical assumptions, the last condition can be also expressed in the language of Lie brackets. This topic is discussed in detail in Sec. 6.2 of \cite{MJ_WR_contact_pmp}.

%%%%%%%%%%%%%%%%%%%%%%%%%%%%%%%%%%%%%%%%%%%%%%%%
\paragraph{Local optimality.}
We shall end our considerations in this section by formulating the main result.  
\begin{theorem}[Local optimality of NSREs]\label{thm:main}
Let $\gamma:[0,T]\ra Q$ be a NSRE. Then there exists a number $\eps>0$ such that for every interval $[a,b]\subset [0,T]$ of length smaller than $\eps$ the curve $\gamma\big|_{[a,b]}:[a,b]\ra Q$ is a minimizing SR geodesic. Moreover, $\gamma\big|_{[a,b]}$ is the unique minimizing geodesic joining its end-points.  
\end{theorem}

Actually, the formulation of Theorem \ref{thm:main} may lead to some confusion as we are considering the restrictions of curves to the interval $[a,b]$, while SR geodesics were defined as curves parametrized by the interval $[0,T]$. Therefore whenever we speak about a restriction $\gamma\big|_{[a,b]}$ of a curve $\gamma:[0,T]\ra Q$ to a sub-interval $[a,b]\subset[0,T]$, we will actually mean a curve $\wt\gamma:[0,T']\ra Q$, being a natural reparametrization of $\gamma\big|_{[a,b]}$ given by $\wt\gamma(t):=\gamma(t-a)$. Here $T':=b-a$ plays the role of the new $T$. Occasionally, we will speak about properties of $\gamma(\cdot)$ invariant under taking such restrictions in the sense of the definition below. 
\begin{definition}\label{def:restriction}
Assume that a curve $\gamma:[0,T]\ra Q$ has a property $A$ ($A$ may depend on $T$). We say that $A$ is \emph{restriction invariant}, if for any sub-interval $[a,b]\subset[0,T]$ the curve $\wt \gamma:[0,T']\ra Q$ defined by $\wt\gamma(t):=\gamma(t-a)$ also has the property $A$. 
\end{definition}
In particular, Theorem \ref{thm:main} implies that the property of being a minimizing SR geodesic is restriction invariant.

%%%%%%%%%%%%%%%%%%%%%%%%%%%%%%%%%%%%%%%%%%%%
%%%%%%%%%%%%%%%%%%%%%%%%%%%%%%%%%%%%%%%%%%%%
\section{The homotopy approach}\label{sec:homotopy}
Since Theorem \ref{thm:main} has a local character, it is enough to restrict our attention to a special case where $\gamma:[0,T]\ra \Omega$, where $\Omega\subset\R^n$ is open, relatively compact and convex. 
Thus from now on we shall be working within the following framework. We assume that  $\gamma:[0,T]\ra\Omega\subset \R^n$ is a trajectory of the control system $(q,u)\mapsto \f(q,u)=\sum_{i=1}^ku^iX_i(q)$ corresponding to a control $u(t)\in \R^k$, where the fields $X_1,\hdots,X_k$ spanning $\D$ are assumed to be $g$-orthonormal, and $u(t)$ is assumed to be normalized by $\| u(t)\|_2\equiv 1$. 

%%%%%%%%%%%%%%%%%%%%%%%%%%%%%%%%%%%%%%%%%%%%%%%
\paragraph{A comparison of energies.}
In order to address the question of local optimality of $\gamma(\cdot)$ one needs to compare the energy of $\gamma(\cdot)$ with the energies of other trajectories of the considered control system. Let now $\wt \gamma:[0,T]\ra \R^n$ be a trajectory of the system corresponding to another control $\wt u(t)=u(t)+\Delta u(t)$ and sharing the same initial-point, i.e., $\wt\gamma(0)=\gamma(0)=q_0$. At the moment we do not impose any restrictions on the end-point $\wt\gamma(T)$. Denote by $\phi(t)$ the $g$-scalar product of $\f_{u(t)}$ and $\f_{\Delta u(t)}$:
\begin{equation}
\label{eqn:phi}
\phi(t):=g(\f_{u(t)},\f_{\Delta u(t)})=\<u(t),\Delta u(t)>_{\R^k}\ ,
\end{equation}
where $\<\cdot,\cdot>_{\R^k}$ denotes the standard scalar product in $\R^k$. Geometrically $\phi(t)$ measures the part of $\f_{\Delta u(t)}$ in the direction of $\f_{u(t)}$, i.e., we can decompose
\begin{equation}
\label{eqn:delta_u_decomposition}
\f_{\Delta u(t)}=\phi(t)\cdot \f_{u(t)}+\f^\perp_{u(t)}\ ,
\end{equation}
where $g(\f^\perp_{u(t)},\f_{u(t)})=0$. Comparing the energies of $\wt\gamma(\cdot)$ and $\gamma(\cdot)$ allows to deduce certain relations between $\phi(t)$, $\|\Delta u(\cdot)\|_{L^2[0,T]}$ and $T$. 

\begin{lemma}\label{lem:basic_estimate}
Assume that $\E(\wt\gamma(\cdot))\leq\E(\gamma(\cdot))$. Then
\begin{equation}
\label{eqn:basic_estimate}
-\int_0^T\phi(t)\dd t\geq \frac 12\|\Delta u(\cdot)\|_{L^2[0,T]}^2\ .
\end{equation}
Moreover, $\|\Delta u(\cdot)\|_{L^2[0,T]}\leq 2\sqrt T$. 
\end{lemma} 
\begin{proof} The proof is just a matter of a simple calculation. By the assumption 
\begin{align*}
\frac 12 \|u(\cdot)\|_{L^2[0,T]}^2=\E(\gamma(\cdot))\geq &\E(\wt\gamma(\cdot))=\frac 12\|u(\cdot)+\Delta u(\cdot)\|_{L^2[0,T]}^2\\
=&\frac 12\|u(\cdot)\|_{L^2[0,T]}^2+\frac 12\|\Delta u(\cdot)\|_{L^2[0,T]}^2+\int_0^T\<u(t),\Delta u(t)>_{\R^k}\dd t\\ 
\overset{\eqref{eqn:phi}}=&\frac 12\|u(\cdot)\|_{L^2[0,T]}^2+\frac 12\|\Delta u(\cdot)\|_{L^2[0,T]}^2+\int_0^T\phi(t)\dd t\end{align*}
and the first part of the assertion follows. To prove the upper estimate for $\|\Delta u(\cdot)\|_{L^2[0,T]}$, note that due to \eqref{eqn:delta_u_decomposition} we can estimate
$$\|\Delta u(t)\|_2^2=g(\f_{\Delta u(t)},\f_{\Delta u(t)})\geq g(\phi(t)\f_{u(t)},\phi(t)\f_{u(t)})=\phi^2(t)\ .$$
Integrating the above inequality we get $\int_0^T\phi^2(t)\dd t\leq \|\Delta u(\cdot)\|_{L^2[0,T]}^2$. Now from \eqref{eqn:basic_estimate} and Cauchy-Schwarz's inequality 
$$\frac 12 \|\Delta u(\cdot)\|_{L^2[0,T]}^2\leq\big|\int_0^T\phi(t)\cdot 1\dd t\big|\overset{\text{C-S}}\leq\left(\int_0^T \phi^2(t)\dd t\right)^{\frac 12}\cdot \left(\int_0^T 1^2\dd t\right)^{\frac 12}\leq \|\Delta u(\cdot)\|_{L^2[0,T]}\cdot \sqrt T\ .$$\end{proof}

%%%%%%%%%%%%%%%%%%%%%%%%%%%%%%%%%%%%%%%%%%%%%%%
\paragraph{Construction of the natural homotopy.}
A crucial idea in our proof of Theorem \ref{thm:main} is to study a homotopy $\gamma_s(\cdot)$ of trajectories in order to compare two trajectories $\gamma_0(\cdot)$ and $\gamma_1(\cdot)$. We study in detail properties of the variation of $\gamma_s(\cdot)$ with respect to the parameter $s$. Such a point of view allows to interpret the rather mysterious geometric conditions \eqref{cond:B_normal} from Theorem~\ref{thm:normal} characterizing NSREs as an information about the above-mentioned variations (see Lemmata~\ref{lem:b_0} and \ref{lem:b_0_nsre}). Now to prove Theorem~\ref{thm:main} it suffices to make an estimate on variations $\|\pa_s\gamma_s(t)-\pa_s\big|_{0}\gamma_s(t)\|_2$ by a method routinely applied to parameter-dependent ODEs.

 Consider a trajectory $\gamma:[0,T]\ra \R^n$ corresponding to the control $u(t)$, and a trajectory $\wt\gamma:[0,T]\ra \R^n$ corresponding to the control $\wt u(t)=u(t)+\Delta u(t)$, both sharing the same initial-point $q_0$. It is natural to construct a family of curves $\gamma_s:[0,T]\ra \R^n$, where for each $s\in[0,1]$ symbol $\gamma_s(\cdot)$ denotes the trajectory of the control system corresponding to the control $u(t)+s\Delta u(t)$ and such that $\gamma_s(0)=\gamma(0)=q_0$. Clearly $\gamma_0(\cdot)=\gamma(\cdot)$ and $\gamma_1(\cdot)=\wt\gamma(\cdot)$, i.e., $\gamma_s(\cdot)$ is a homotopy joining $\gamma(\cdot)$ and $\wt\gamma(\cdot)$, see Figure \ref{fig:idea}. We shall call it the \emph{natural homotopy} between $\gamma(\cdot)$ and $\wt \gamma(\cdot)$.

In our considerations studying the \emph{variation} of $\gamma_s(\cdot)$ with respect to the parameter $s$ will be of crucial importance. We denote
$$b_s(t):=\pa_s \gamma_s(t)\ .$$
At $s=0$ this variation can be expressed in terms of the TD-flow $\F_{t\tau}$ of the field $\f_{u(t)}$.\footnote{Actually the results of Lemma \ref{lem:b_0} can be easily extended to describe $b_s(t)$ for any $s\in[0,1]$ in terms of the TD-flow of the field $\f_{u(t)+s\Delta u(t)}$.} Moreover, if $\gamma(\cdot)$ satisfies the regularity condition \eqref{cond:B_normal1} of Theorem \ref{thm:normal}, we can also estimate the $\f_{u(t)}(\gamma(t))$-component of $b_0(t)$.
\begin{lemma}\label{lem:b_0}
For every $t\in[0,T]$, we have
\begin{equation}
\label{eqn:b_0}
b_0(t)=\int_0^t \T\F_{t\tau}\left[\f_{\Delta u(\tau)}(\gamma(\tau))\right]\dd\tau\ .
\end{equation}
Moreover, if $\gamma(\cdot)$ is $C^1$-smooth with an ACB derivative, and the control vector field $\f_{u(t)}$ is normalized as $g(\f_{u(t)},\f_{u(t)})=1$, then we may represent $b_0(t)$ as
\begin{equation}
\label{eqn:b_0_decomposition}
b_0(t)=\int_0^t \phi(\tau)\dd\tau \cdot \dot\gamma(t)+\wh b_0(t)\ ,
\end{equation}
where $\wh b_0(t)\in \F_\bullet (\Dperp)_{\gamma(t)}$ and $\phi(t)=\<u(t),\Delta u(t)>_{\R^k}$. 
\end{lemma}
\begin{proof} By definition, $\pa_t \gamma_s(t)\overset{\text{a.e.}}=\f_{u(t)+s\Delta u(t)}(\gamma_s(t))=\sum_i\left(u^i(t)+s\Delta u^i(t)\right)X_i(\gamma_s(t))$. Thus $\gamma_s(t)$ is a solution (in the sense of Caratheodory) of an ODE of the form $\dot x=f(x,t,s)$, where $f$ depends smoothly on the position $x$ and the parameter $s$, and measurably on the time $t$. The standard theory of existence and uniqueness of solutions of ODEs and their regularity with respect to the initial conditions and parameters can be directly applied to such ODEs. Basically the Piccard's method works equally well in the measurable setting as in the continuous one (see \cite{Bressan_Piccoli_2004}). It follows that $\gamma_s(t)$ is differentiable with respect to $s$, i.e., the variation $b_s(t)$ is well-defined, and, moreover, that $b_s(t)$ is a solution (in the sense of Caratheodory) of variation of the initial equation for $\gamma_s(t)$ (and the mixed partial derivatives commute a.e.). That is, 
\begin{equation}
\label{eqn:b_dot}
\begin{split}
\pa_t b_s(t)=&\pa_t\pa_s \gamma_s(t)\overset{\text{a.e.}}=\pa_s\pa_t \gamma_s(t)=\pa_s\left[\sum_i\left(u^i(t)+s\Delta u^i(t)\right)X_i(\gamma_s(t))\right]\\
=&\sum_i\Delta u^i(t) X_i(\gamma_s(t))+\sum_i\left(u^i(t)+s\Delta u^i(t)\right)\frac{\pa X_i}{\pa q}\bigg|_{\gamma_s(t)}b_s(t)\ .
\end{split}
\end{equation}
In particular, $b_0(t)$ is a solution of the linear non-homogeneous time-varying ODE
\begin{equation}
\label{eqn:b_0_dot}
\pa_t b_0(t)=\f_{\Delta u(t)}(\gamma(t))+\pa_q \f_{u(t)}(q)\Big|_{\gamma(t)}b_0(t)\ ,
\end{equation}
with the initial condition $b_0(0)=\pa_s\big|_0\gamma_s(0)=\pa_s\big|_0q_0=0$, and can be easily found by the standard methods. Indeed, consider first the homogeneous equation
$$\dot Y(t)=\pa_q \f_{u(t)}(q)\big|_{\gamma(t)}Y(t) $$
for a matrix $Y(t)$ with the initial condition $Y(0)=\Id_{\T_{\gamma(0)}\R^n}$.
Clearly, this is simply the variational equation of $\dot \gamma(t)=\f_{u(t)}(\gamma(t))$. Thus the solution of the homogeneous equation is provided by the family of tangent maps $t\mapsto \T \F_{t0}$ given by the TD-flow $\F_{t\tau}$ of $\f_{u(t)}$ evaluated at $\tau=0$. Now formula \eqref{eqn:b_0} follows directly from the standard theorem describing solutions of linear non-homogeneous time-varying ODEs (based on the variation of constant method).\medskip

Assume now that $t\mapsto \dot \gamma(t)=\f_{u(t)}(\gamma(t))$ is ACB. Then its derivative is well-defined a.e. and satisfies
$$\ddot \gamma(t)=\sum_i \dot u^i(t)X_i(\gamma(t))+\sum_i u^i(t)\frac{\pa X_i}{\pa q}\Big|_{\gamma(t)}\dot \gamma(t)\ .$$
This equation is of the same nature as \eqref{eqn:b_0_dot} and thus its solution (treating $\dot \gamma(\tau)$ as an initial condition at $t=\tau$) reads as
\begin{equation}
\label{eqn:dot_gamma}
\dot \gamma(t)=\T\F_{t\tau}\left[\dot\gamma(\tau)\right]+\int_\tau^t \T\F_{ts}\left[\f_{\dot u(s)}(\gamma(s))\right]\dd s\,.\end{equation}
Note that $\f_{\dot u(s)}(\gamma(s)):=\sum_i\dot u^i(s)X_i(\gamma(s))$ is $g$-orthogonal to $\dot\gamma(s)=\sum_i u^i(s)X_i(\gamma(s))$, which follows easily from the normalization condition.
 
Using the above formula, the decomposition \eqref{eqn:delta_u_decomposition}, and the linearity of the tangent map we get 
\begin{align*}
b_0(t)\overset{\eqref{eqn:b_0}}=&\int_0^tT \F_{t\tau}\left[\f_{\Delta u(\tau)}(\gamma(\tau))\right]\dd\tau\\
\overset{\eqref{eqn:delta_u_decomposition}}=&\int_0^t\F_{t\tau}\left[\phi(\tau)\cdot \dot\gamma(\tau)+\f_{ u(\tau)}^\perp(\gamma(\tau))\right]\dd\tau\\
=&\int_0^t\phi(\tau)\cdot\T\F_{t\tau}\left[\dot\gamma(\tau)\right]\dd \tau+\int_0^tT \F_{t\tau}\left[\f_{u(\tau)}^\perp(\gamma(\tau))\right]\dd\tau\\
\overset{\eqref{eqn:dot_gamma}}=&\int_0^t\phi(\tau)\cdot \left[\dot\gamma(t)-\int_\tau^t\T\F_{ts}\left[\f_{\dot u(s)}(\gamma(s))\right]\dd s\right]\dd \tau+\int_0^tT \F_{t\tau}\left[\f_{u(\tau)}^\perp(\gamma(\tau))\right]\dd\tau\\
=&\int_0^t\phi(\tau)\dd \tau \cdot \dot\gamma(t)-\int_0^t \phi(\tau)\int_\tau^t\T\F_{ts}\left[\f_{\dot u(s)}(\gamma(s))\right]\dd s\dd \tau+\int_0^t\T\F_{ts}\left[\f_{u(s)}^\perp(\gamma(s))\right]\dd s\ .
\end{align*}
To check that the element 
$$\wh b_0(t):=-\int_0^t \phi(\tau)\int_\tau^t\T\F_{ts}\left[\f_{\dot u(s)}(\gamma(s))\right]\dd s\dd \tau+\int_0^t\T \F_{ts}\left[\f_{u(s)}^\perp(\gamma(s))\right]\dd s$$
belongs to $\F_\bullet(\Dperp)_{\gamma(t)}$ note that, for each $s\in[0,t]$, the element $\T \F_{ts}\left[\f_{u(s)}^\perp(\gamma(s))\right]$ as well as the element  $\T\F_{ts}\left[\f_{\dot u(s)}(\gamma(s))\right]$  belong to $\F_\bullet (\Dperp)_{\gamma(t)}$ as the images of vectors orthogonal to $\dot\gamma(t)$ under the TD-flow $\T \F_{ts}$. Thus also integration of such terms gives an element of $\F_\bullet(\Dperp)_{\gamma(t)}$.  This ends the proof. \end{proof} 

Formula \eqref{eqn:b_0_decomposition} in the above result provides the splitting of $b_0(t)$ into $\dot\gamma(t)$- and $\F_\bullet(\Dperp)_{\gamma(t)}$-parts. If, additionally, $\gamma(\cdot)$ is a NSRE, then condition \eqref{cond:B_normal2} from Theorem \ref{thm:normal} guarantees that the velocity $\dot\gamma(t)$ does not belong to $\F_\bullet(\Dperp)_{\gamma(t)}$, which allows to make the following estimate.

\begin{lemma}\label{lem:b_0_nsre}
If $\gamma:[0,T]\ra \R^n$ is a NSRE, then there exist a number $c>0$, depending only on $\gamma(\cdot)$, such that for every $t\in[0,T]$ 
\begin{equation}
\label{eqn:b_0_estimate}
\|b_0(t)\|_2\geq c\cdot|\int_0^t\phi(\tau)\dd \tau|\ .
\end{equation}  
Moreover, the above inequality (with the same constant $c$) is restriction invariant in the sense of Definition~\ref{def:restriction}. 
\end{lemma} 
\begin{proof}
The norm of any $\R^n$-vector of the form $v=\alpha\cdot v_0+w$, where $v_0\in\R^n$ is a given vector, $\alpha$ is a real number, and $w$ belongs to a given linear subspace $W\subset \R^n$ can be estimated from below as follows
$$\|v\|_2\geq |\alpha|\cdot\|v_0\|_2\cdot|\sin \theta|\ ,$$
where $\theta$ is the angle between the vector $v_0$ and the subspace $W$. Applying this simple estimate to formula~\eqref{eqn:b_0_decomposition} (recall that, by Theorem \ref{thm:normal}, a NSRE is sufficiently regular for \eqref{eqn:b_0_decomposition} to hold) we conclude that for every $t\in[0,T]$ 
$$ \|b_0(t)\|_2\geq |\int_0^t\phi(t)\dd t|\cdot\|\f_{u(t)}(\gamma(t))\|_2\cdot |\sin\theta (t)|\ ,$$
where $\theta(t)$ is the angle between $\dot\gamma(t)=\f_{u(t)}(\gamma(t))$ and the space $\F_\bullet(\Dperp)_{\gamma(t)}$. 
By the results of Theorem~\ref{thm:main} we know that $\theta(t)\neq 0$. Now it is enough to take $c=\inf_{t\in[0,T]}|\sin\theta(t)|\cdot \inf_{t\in[0,T]}\|\f_{u(t)}(\gamma(t))\|_2$ and make sure that $c$ defined in such a way is non-zero. This follows easily from the fact that both assignments $t\mapsto \f_{u(t)}(\gamma(t))$ and $t\mapsto \F_\bullet(\Dperp)_{\gamma(t)}$ are continuous with respect to $t$ and defined over the compact set $[0,T]$. Indeed, $t\mapsto \f_{u(t)}(\gamma(t))$ is continuous by Theorem \ref{thm:normal} and $t\mapsto \F_\bullet(\Dperp)_{\gamma(t)}$ as the image of $\F_\bullet(\Dperp)_{\gamma(0)}$ under $t\mapsto \T \F_{t0}(\cdot)$ which,  by Lemma \ref{lem:flows}, is continuous. Here we use the $\T \F_{t0}$-invariance of $\F_\bullet(\Dperp)$.
\end{proof}  

 For future purposes we will also need to estimate the difference of $b_s(t)-b_0(t)$ for arbitrary $s\in[0,1]$. This can be achieved provided that the whole natural homotopy $\gamma_s(\cdot)$ lies in a given relatively compact and convex neighborhood of $\gamma(\cdot)$. 

\begin{lemma}\label{lem:estimate_delta_b}
Let $\gamma:[0,T]\ra \Omega\subset\R^n$ be a trajectory of the considered control system (not necessarily a NSRE) as above. There exists a number $\delta>0$, depending only on $\Omega$ and $\gamma(\cdot)$ such that for any $\wt\gamma(\cdot)$ if $\|\Delta u(\cdot)\|_{L^2[0,T]}\leq \delta$, then the whole homotopy $\gamma_s(\cdot)$ lies inside $\Omega$. 

Assuming that the homotopy $\gamma_s(\cdot)$ lies in $\Omega$ and that $\E(\wt\gamma(\cdot))\leq \E(\gamma(\cdot))$, there exists a continuous function $\xi:\R\ra\R_+$ such that for every $t\in[0,T]$ and every $s\in[0,1]$ we can estimate
\begin{equation} p
\label{eqn:estimate_delta_b}
\|b_s(t)-b_0(t)\|_2\leq T \xi(T)\|\Delta u(\cdot)\|^2_{L^2[0,T]}\ .
\end{equation} 
Moreover, the number $\delta$ and the above inequality (with the same function $\xi$) are restriction invariant in the sense of Definition \ref{def:restriction}.
\end{lemma} 
The proof is elementary (it uses only Gr\"{o}nwall's inequality and the mean value theorem), yet rather long (several estimations need to be done). To improve the clarity of the exposition we move it to Appendix \ref{app:technical}.  

Now to prove Theorem \ref{thm:main} it is enough to put together estimates obtained in Lemmata \ref{lem:basic_estimate}, \ref{lem:b_0_nsre}, and~\ref{lem:estimate_delta_b}. 

%%%%%%%%%%%%%%%%%%%%%%%%%%%%%%%%%%%%%%%%%%%%%%%
\paragraph{Proof of Theorem \ref{thm:main}.}
Let $\gamma:[0,T]\ra \Omega\subset\R^n$ be a NSRE lying in a relatively compact and convex open set $\Omega$. Let  $c$, $\delta$ and $\xi(T)$ be as in Lemmata \ref{lem:b_0_nsre} and \ref{lem:estimate_delta_b}. We choose a number $\eps>0$ satisfying the conditions
\begin{equation}
\label{eqn:cond_eps}
2\sqrt{\eps}<\delta\quad\text{and}\quad \sup_{T\in[0,\eps]}T \xi(T)<\frac c2\ .
\end{equation}
We claim that any restriction of $\gamma(\cdot)$ to an interval $[a,b]\subset[0,T]$ of length smaller than $\eps$ is a minimizing SR geodesic. 

Without any loss of generality we may assume that $[a,b]=[0,T']\subset[0,T]$ with $T'<\eps$ (we can always reparametrize $\gamma\big|_{[a,b]}(\cdot)$ by $t\mapsto t-a$). Let us now choose any trajectory $\wt\gamma:[0,T']\ra \R^n$ sharing the same initial point as $\gamma(\cdot)$ and such that $\E(\wt\gamma(\cdot))<\E(\gamma\big|_{[0,T']}(\cdot))$. We denote the control corresponding to $\wt\gamma(\cdot)$ by $u(t)+\Delta u(t)$ and build the natural homotopy $\gamma_s(\cdot)$ joining $\gamma\big|_{[0,T']}(\cdot)$ with $\wt\gamma(\cdot)$ as described in this section. By Lemma \ref{lem:basic_estimate} we know that $\|\Delta u(\cdot)\|_{L^2[0,T']}\leq 2\sqrt{T'}$. Since $2\sqrt{T'}<2\sqrt{\eps}\overset{\eqref{eqn:cond_eps}}<\delta$, Lemma \ref{lem:estimate_delta_b} implies that $\wt\gamma(\cdot)$ and the whole related homotopy $\gamma_s(\cdot)$ lie inside $\Omega$. Moreover, the difference $b_s(T')-b_0(T')$ can be estimated by \eqref{eqn:estimate_delta_b}. Finally, by the results of Lemmata \ref{lem:basic_estimate} and \ref{lem:b_0_nsre}, we estimate:
\begin{equation}
\|b_0(T')\|_2\overset{\eqref{eqn:b_0_estimate}}\geq c\cdot\big|\int_0^{T'}\phi(\tau)\dd \tau\big|\overset{\eqref{eqn:basic_estimate}}\geq\frac c2\cdot\|\Delta u(\cdot)\|^2_{L^2[0,T']}\ .
\label{eqn:estimate_b_0}
\end{equation}
In consequence we get
\begin{align*}
\|\wt\gamma(T')-\gamma(T')\|_2&=\|\gamma_1(T')-\gamma_0(T')\|_2=\big\|\int_0^1\pa_s\gamma_s(T')\dd s\big\|_2=\big\|\int_0^1b_s(T')\dd s\big\|_2\\
&=\big\|\int_0^1b_s(T')-b_0(T')+b_0(T')\dd s\big\|_2\\
&\geq\big\|\int_0^1 b_0(T')\dd s\big\|_2-\big\|\int_0^1b_s(T')-b_0(T')\dd s\big\|_2\\
&\geq\|b_0(T')\|_2-\int_0^1\|b_s(T')-b_0(T')\|_2\dd s \\
&\overset{\text{\eqref{eqn:estimate_b_0} and \eqref{eqn:estimate_delta_b}}}\geq\left(\frac c2-T'\xi(T')\right)\|\Delta u(\cdot)\|^2_{L^2[0,T']}\ .
\end{align*}
Due to \eqref{eqn:cond_eps} we conclude that if $\Delta u$ does not satisfy $\Delta u(\cdot)= 0$ a.e., then $\|\wt\gamma(T')-\gamma(T')\|_2>0$, i.e., no trajectory with the energy smaller then the energy of $\gamma\big|_{[0,T']}(\cdot)$ can have the same end-points as $\gamma\big|_{[0,T']}(\cdot)$. This ends the proof. \qed

\appendix
%%%%%%%%%%%%%%%%%%%%%%%%%%%%%%%%%%%%%%%%%%%%%%%%%%%%%
%%%%%%%%%%%%%%%%%%%%%%%%%%%%%%%%%%%%%%%%%%%%%%%%%%%%%
\section{Technical results}\label{app:technical}

In this part we formulate various technical results needed in the course of this work.

%%%%%%%%%%%%%%%%%%%%%%%%%%%%%%%%%%%%%%%%%%%%%%%%%%
\paragraph{Flows of control vector fields.}

\begin{lemma}\label{lem:flows}
Within the notation from Section \ref{sec:nsre}, let $\gamma:[0,T]\ra Q$ be a trajectory of a control system $f:Q\times\R^k\ra \T Q$ corresponding to a given locally bonded measurable control $u:[0,T]\ra\R^k$ and let $\F_{t\tau}(\cdot)$ be the TD-flow of the related time-dependent vector field $\f_{u(t)}$. Then $\F_{t\tau}(\cdot)$ is a well-defined family of local diffeomorphisms in a neighborhood of $\gamma(\cdot)$. Moreover, for any fixed $\tau\in[0,T]$, we have:
\begin{enumerate}[(i)]
	\item The assignment $t\mapsto \F_{t\tau}(q)$ is ACB whenever defined.
	\item The assignment $t\mapsto \T \F_{t\tau}(\cdot):\T_qQ\ra \T_{\F_{t\tau}(q)}Q$ is ACB whenever $\F_{t\tau}(q)$ is defined. 
\end{enumerate}  
\end{lemma}
An analogous result is valid also in a more general context of Caratheodory vector fields. The reader may consult Sec. 2 and Appendix of \cite{MJ_WR_contact_pmp} for a detailed discussion. 

%%%%%%%%%%%%%%%%%%%%%%%%%%%%%%%%%%%%%%%%%%%%%%%%%%
\paragraph{Gr\"{o}nwall's inequality.}
Now we will formulate a version of the well-known \emph{Gr\"{o}nwall's inequality} (see e.g., \cite{Evans}), which is extensively used in various estimations needed in this paper. 

\begin{lemma}[Gr\"onwall's inequality]\label{lem:gronwall}
Let $\alpha,\beta:\R\ra\R$ be non-negative locally bounded measurable maps and let $x:[0,T]\ra\R^n$ be an absolutely continuous curve in $\R^n$ such that $x(0)=0$ and
\begin{equation}
\label{eqn:norm_assumptions}
\|\dot x(t)\|_2\leq \alpha(t)+\beta(t)\|x(t)\|_2 \quad \text{for $t\in[0,T]$  a.e..}
\end{equation} 
Then
\begin{equation}
\label{eqn:norm_estimate}
\|x(t)\|_2\leq e^{\int_0^t\beta(\tau)\dd\tau}\cdot\int_0^t\alpha(\tau)\dd\tau \quad\text{for every $t\in[0,T]$.}
\end{equation}
\end{lemma}

%%%%%%%%%%%%%%%%%%%%%%%%%%%%%%%%%%%%%%
\paragraph{A proof of Lemma \ref{lem:estimate_delta_b}.}
We shall end this technical part by providing a proof of Lemma \ref{lem:estimate_delta_b}. The reasoning is elementary, yet quite long, so we shall divide it into parts. 

Let $\gamma:[0,T]\ra \Omega\subset\R^n$ be the trajectory of the control system $(q,u)\mapsto f(q,u)=\sum_iu^iX_i(q)$, corresponding to a normalized control $u(t)$ and contained in a given relatively compact and convex open set $\Omega$. Consider another trajectory $\wt\gamma:[0,T]\ra \R^n$ corresponding to the control $u(t)+\Delta u(t)$ and denote by $\gamma_s:[0,T]\ra \R^n$ for $s\in[0,1]$ the natural homotopy related with $\Delta u(t)$. Note that since $\|u(t)\|_2=1$ we have 
\begin{equation}
\label{eqn:norms} 
\|u(t)\|_1\leq \sqrt{k}\cdot\|u(t)\|_2=\sqrt{k}\ .
\end{equation} 

Let $(q^1,\hdots,q^n)$ be coordinates in $\R^n$. Choose finite non-negative numbers $C_0$, $C_1$, $C_2$, and $C_3$ such that 
\begin{equation}
\label{eqn:estimates}
\begin{split}
&\|X_i(q)\|_2\leq C_0,\quad\|\frac{\pa X_i}{\pa q^a}\Big|_q \|_2\leq C_1,\quad\\
 &\|X_i(q)-X_i(q')\|_2\leq C_2\|q-q'\|_2,\quad\text{and}\quad\|\frac{\pa X_i}{\pa q^a}\Big|_{q'}-\frac{\pa X_i}{\pa q^a}\Big|_q\|_2\leq C_3\|q'-q\|_2
\end{split}
\end{equation} 
for every $q,q'\in \overline\Omega$, $i=1,\hdots, k$ and $a=1,\hdots,n$. If the $\|\cdot\|_2$-norm in \eqref{eqn:estimates} is repalced by the $\|\cdot\|_1$-norm then the existence of $C_0$ and $C_1$ is clear, as the vector fields $X_i$ are smooth and $\overline\Omega$ is compact. Similarly, the existence of $C_2$ and $C_3$ follows easily from the mean value theorem (note that $\Omega$ is convex), smoothness of $X_i$ and compactness of $\overline \Omega$. Since all norms in an Euclidean space are equivalent, estimates \eqref{eqn:estimates} (with modified constants) hold also for $\|\cdot\|_2$-norm.

At first we shall estimate the norm of the difference $\gamma_s(t)-\gamma_0(t)$ which will allow us to prove the first part of the assertion of Lemma \ref{lem:estimate_delta_b}. 
\begin{lemma}[Estimate of $\gamma_s(t)-\gamma_0(t)$]\label{lem:estimate_delta_q}
Assuming that the whole homotopy $\gamma_s(\cdot)$ lies in $\Omega$, there exists a continuous function $\zeta:\R\ra\R_+$ such that for every $t\in[0,T]$ and every $s\in[0,1]$ the following inequality holds
\begin{equation}
\label{eqn:estimate_delta_q}
\lVert \gamma_s(t)-\gamma_0(t)\rVert_2\leq \sqrt{T}\zeta(T)\|\Delta u(\cdot)\|_{L^2[0,T]}\ .
\end{equation}
\end{lemma}
\begin{proof}
By definition, for every $s\in[0,1]$ we have 
$$\dot \gamma_s(t)=\f_{u(t)+s\Delta u(t)}(\gamma_s(t))=\sum_i \left[u^i(t)+s\Delta u^i(t)\right]X_i(\gamma_s(t))\ .$$
Thus 
$$
\dot \gamma_s(t)-\dot \gamma_0(t)=\sum_i u^i(t)\left[X_i(\gamma_s(t))-X_i(\gamma_0(t))\right]+ s\sum_i \Delta u^i(t)X_i(\gamma_s(t))\ .
$$
Since  $\gamma_s(\cdot)$ lies in $\Omega$ it follows that 
\begin{align*}
\|\dot \gamma_s(t)-\dot \gamma_0(t)\|_2\leq&\sum_i |u^i(t)|\cdot\|X_i(\gamma_s(t))-X_i(\gamma_0(t))\|_2+ s\sum_i |\Delta u^i(t)|\cdot\|X_i(\gamma_s(t))\|_2 \\
\overset{\eqref{eqn:estimates}}\leq&\|u(t)\|_1\cdot C_2\cdot\|\gamma_s(t)-\gamma_0(t)\|_2+ \|\Delta u(t)\|_1\cdot C_0\\
\overset{\eqref{eqn:norms}}\leq&\sqrt{k}C_2\cdot\|\gamma_s(t)-\gamma_0(t)\|_2+ \|\Delta u(t)\|_2\sqrt{k}C_0\ .
\end{align*}
Now we can apply Gr\"{o}nwall's inequality, as formulated in Lemma \ref{lem:gronwall} (note that $\gamma_s(0)-\gamma_0(0)=0$), to get the estimate
$$\|\gamma_s(t)-\gamma_0(t)\|_2\leq e^{\sqrt{k} C_2 t}\sqrt kC_0\int_0^t\|\Delta u(\tau)\|_2\dd \tau.$$
By Cauchy-Schwarz inequality and $t\leq T$, we conclude
\begin{equation}
\label{eqn:holder_u_L2}
\int_0^t\|\Delta u(\tau)\|_2\dd \tau\leq \left(\int_0^t\|\Delta u(\tau)\|^2_2\dd \tau\right)^{\frac 12}\cdot\left(\int_0^t 1\dd\tau\right)^{\frac 12}\overset{(t\leq T)}\leq \|\Delta u(\cdot)\|_{L^2[0,T]}\cdot\sqrt{T}\ ,
\end{equation}
and thus \eqref{eqn:estimate_delta_q} holds with $\zeta(T):=e^{\sqrt{k} C_2 T}\sqrt kC_0$. \end{proof}

As a simple corollary from the above reasoning we obtain the following result.
\begin{lemma}
There exists a number $\delta >0$ such that if $\|\Delta u(\cdot)\|_{L^2[0,T]}\leq \delta$, then the whole homotopy $\gamma_s(\cdot)$ lies in $\Omega$. 
\end{lemma}
\begin{proof}Let $\eta$ be a number such that $\Omega_\eta:=\{q\in\R^n\ |\ \exists_{t\in[0,T]}\quad\|q-\gamma(t)\|_2<\eta \}$,  the set of points within the $\eta$-distance from $\gamma(\cdot)$, is entirely contained in $\Omega$. Clearly $\eta>0$ since the image of $\gamma(\cdot)$ is compact and contained in $\Omega$. We claim that, for $\zeta:\R\ra\R_+$ as in the assertion of Lemma \ref{lem:estimate_delta_q}, if 
$$\sqrt{T}\zeta(T)\|\Delta u(\cdot)\|_{L^2[0,T]}\leq \frac{\eta}2\ ,$$ 
then the whole homotopy $\gamma_s(\cdot)$ lies in $\Omega$, i.e., $\delta=\frac{\eta}{2\sqrt{T}\zeta(T)}$.

Indeed, assume the contrary and let $t_0\in[0,T]$ be the supremum of these $t\in[0,T]$ such that $\gamma_s(\tau)\in \Omega$ for every $\tau\in[0,t]$ and every $s\in[0,1]$. Clearly $t_0$ is well defined as $\gamma_s(0)=q_0\in\Omega$ for every $s\in[0,1]$. By the continuity argument $\gamma_s(t_0)\in \overline\Omega$ for every $s\in[0,1]$ and also $\gamma_{s_0}(t_0)\in \partial \Omega$ for some $s_0\in[0,1]$ as we assumed that the homotopy $\gamma_s(\cdot)$ is not entirely contained in $\Omega$. (Basically the number $t_0$ captures the first moment at which one of the curves $\gamma_s(t)$ exists the domain $\Omega$.) 

It follows from the proof of Lemma \ref{lem:estimate_delta_q} that the estimate \eqref{eqn:estimate_delta_q} holds for these $t$ and $s$ such that $\gamma_\sigma(\tau)\in \overline \Omega$ for every $\sigma\in[0,s] $ and every $\tau\in [0,t]$. In particular, 
$$\|\gamma_{s_0}(t_0)-\gamma_0(t_0)\|_2\leq\frac \eta 2\ ,$$ 
and thus $\gamma_{s_0}(t_0)\in \Omega_{\frac\eta 2}\subset \operatorname{int} \Omega$. This contradicts our previous assumption. \end{proof}

With similar methods that were used for proving Lemma \ref{lem:estimate_delta_q} one can estimate the norm of $b_s(t)$. 
\begin{lemma}[Estimate of $b_s(t)$]\label{lem:estimate_b}
Assuming that the homotopy $\gamma_s(\cdot)$ lies in $\Omega$ and that $\E(\wt\gamma(\cdot))\leq \E(\gamma(\cdot))$, there exists a continuous function $\psi:\R\ra\R_+$ such that for every $t\in[0,T]$ and $s\in[0,1]$ the following inequality holds
\begin{equation}
\label{eqn:estimate_b_s}
\|b_s(t)\|_2\leq \sqrt{T}\psi(T)\|\Delta u(\cdot)\|_{L^2[0,T]}\ .
\end{equation}
\end{lemma}
\begin{proof}
Note that by \eqref{eqn:b_dot}
$$
\dot b_s(t)=\sum_i\left[u^i(t)+s\Delta u^i(t)\right]\sum_a\frac{\pa X_i}{\pa q^a}\Big|_{\gamma_s(t)}(b_s)^a(t)+\sum_i\Delta u^i(t)X(\gamma_s(t))\ .
$$ 
Now we can estimate 
\begin{align*}
\|\dot b_s(t)\|_2\overset{\eqref{eqn:estimates}}\leq&\left(\|u(t)\|_1+\|\Delta u(t)\|_1\right)\sum_aC_1|(b_s)^a(t)|+\|\Delta u(t)\|_1\cdot C_0 \\
\overset{\eqref{eqn:norms}}\leq&\sqrt{k}\left(1+\|\Delta u(t)\|_2\right) n C_1\|b_s(t)\|_2+C_0\sqrt{k}\|\Delta u(t)\|_2\ .
\end{align*}
By Gr\"{o}nwall's inequality (Lemma \ref{lem:gronwall}) (note that $b_s(0)=\pa_s\gamma_s(0)=\pa_sq_0=0$), we get
\begin{align*}\|b_s(t)\|_2\leq &e^{C_1\sqrt{k}n (t+\int_0^t\|\Delta u(\tau)\|_2\dd\tau)}C_0\sqrt{k}\int_0^t\|\Delta u(\tau)\|_2\dd\tau\\
\overset{\eqref{eqn:holder_u_L2}}\leq& e^{C_1\sqrt{k}n (t+\int_0^t\|\Delta u(\tau)\|_2\dd\tau)}C_0\sqrt{k}\sqrt{T}\|\Delta u(\cdot)\|_{L^2[0,T]}\ .
\end{align*}
Now, as $\E(\wt\gamma(\cdot))\leq \E(\gamma(\cdot))$ we have $\int_0^t\|\Delta u(\tau)\|_2\dd\tau\overset{\eqref{eqn:holder_u_L2}}\leq\sqrt{T}\|\Delta u(\cdot)\|_{L^2[0,T]}\overset{\text{Lem. \ref{lem:basic_estimate}}}\leq 2T$. We conclude that the assertion holds with $\psi(T):=e^{3C_1\sqrt{k}n T}C_0\sqrt{k}$.
\end{proof}

Finally we may calculate an estimate of $\|b_s(t)-b_0(t)\|_2$. Using the formula for $\dot b_s(t)$ from the proof of the above lemma we get
\begin{align*}
\dot b_s(t)-\dot b_0(t)= &\sum_iu^i(t)\sum_a\left[\frac{\pa X_i}{\pa q^a}\Big|_{\gamma_s(t)}(b_s)^a(t)-\frac{\pa X_i}{\pa q^a}\Big|_{\gamma_0(t)}(b_0)^a(t)\right]\\
&+s\sum_i\Delta u^i(t)\sum_a\frac{\pa X_i}{\pa q^a}\Big|_{\gamma_s(t)}(b_s)^a(t)+\sum_i\Delta u^i(t)\left[X_i(\gamma_s(t))-X_i(\gamma_0(t))\right]\\
=&\sum_iu^i(t)\sum_a\frac{\pa X_i}{\pa q^a}\Big|_{\gamma_s(t)}\left[(b_s)^a(t)-(b_0)^a(t)\right]\\
&-\sum_iu^i(t)\sum_a\left[\frac{\pa X_i}{\pa q^a}\Big|_{\gamma_s(t)}-\frac{\pa X_i}{\pa q^a}\Big|_{\gamma_0(t)}\right](b_0)^a(t)\\
&+s\sum_i\Delta u^i(t)\sum_a\frac{\pa X_i}{\pa q^a}\Big|_{\gamma_s(t)}(b_s)^a(t)+\sum_i\Delta u^i(t)\left[X_i(\gamma_s(t))-X_i(\gamma_0(t))\right]
\end{align*} 
Now due to \eqref{eqn:estimates} we can estimate
\begin{align*}
\|\dot b_s(t)-\dot b_0(t)\|_2\leq&\|u(t)\|_1\cdot nC_1\|b_s(t)-b_0(t)\|_2+\|u(t)\|_1\cdot nC_3\|\gamma_s(t)-\gamma_0(t)\|_2\cdot\|b_0(t)\|_2\\
&+\|\Delta u(t)\|_1\cdot nC_1\|b_s(t)\|_2+\|\Delta u(t)\|\cdot C_2\|\gamma_s(t)-\gamma_0(t)\|_2
\end{align*}
By \eqref{eqn:norms} and the results of Lemmata \ref{lem:estimate_delta_q} and \ref{lem:estimate_b} we get
\begin{align*}
\|\dot b_s(t)-\dot b_0(t)\|_2\leq&\sqrt{k}nC_1\|b_s(t)-b_0(t)\|_2+\sqrt{k}nC_3T\zeta(T)\psi(T)\|\Delta u(\cdot)\|^2_{L^2[0,T]}\\
&+\sqrt{k}\|\Delta u(t)\|_2 nC_1\sqrt{T}\psi(T)\|\Delta u(\cdot)\|_{L^2[0,T]}+\sqrt{k}\|\Delta u(t)\|_2C_2\sqrt{T}\psi(T)\|\Delta u(\cdot)\|_{L^2[0,T]}\ .
\end{align*}
Thus Gr\"{o}nwall's inequality (Lemma \ref{lem:gronwall}) (again $b_s(0)-b_0(0)=0$) gives us the following estimate:
\begin{align*}
\|b_s(t)-b_0(t)\|_2\leq e^{\sqrt{k}nC_1 t}&\left[\sqrt{k}nC_3T\zeta(T)\psi(T)\|\Delta u(\cdot)\|^2_{L^2[0,T]}\int_0^t 1\dd\tau\right.\\
&+\left.\sqrt{k}\sqrt{T}(nC_1\psi(T)+C_2\zeta(T))\|\Delta u(\cdot)\|_{L^2[0,T]}\int_0^t\|\Delta u(\tau)\|_2\dd\tau\right]
\end{align*}
Again using estimate \eqref{eqn:holder_u_L2} we arrive at inequality \eqref{eqn:estimate_delta_b} with 
$$\xi (T):=e^{\sqrt{k}nC_1 T}\left[\sqrt{k}nC_3T\zeta(T)\psi(T)+\sqrt{k}(nC_1\psi(T)+C_2\zeta(T))\right].$$
This ends the proof of Lemma \ref{lem:estimate_delta_b}. \qed	
	
\bibliographystyle{amsalpha}
\bibliography{bibl}

\end{document}